\documentclass[12pt]{amsart}
\usepackage{amssymb,amsthm,amsmath}
\usepackage[all]{xy}
\usepackage{cite}

\theoremstyle{definition}
\newtheorem{defn}{Definition}[section]
\newtheorem{setup}[defn]{Setup}
\theoremstyle{remark}
\newtheorem{remark}{Remark}
\theoremstyle{plain}
\newtheorem{thm}[defn]{Theorem}
\newtheorem{prop}[defn]{Proposition}
\newtheorem{cor}[defn]{Corollary}
\newtheorem{lemma}[defn]{Lemma}
\newtheorem*{thma}{Theorem}

\newcommand{\V}{\mathbb{V}}

\newcommand{\cO}{\mathcal{O}}

\newcommand{\cJ}{\mathcal{J}}

\newcommand{\cF}{\mathcal{F}}
\newcommand{\cL}{\mathcal{L}}
\newcommand{\cM}{\mathcal{M}}

\newcommand{\cI}{\mathcal{I}}

\newcommand{\fa}{\mathfrak{a}}

\newcommand{\fb}{\mathfrak{b}}

\newcommand{\Q}{\mathbb Q}
\newcommand{\R}{\mathbb R}

\newcommand{\C}{\mathbb C}

\newcommand{\codim}{\textrm{codim}}

\newcommand{\Supp}{\textnormal{Supp }}

\newcommand{\Jac}{\textnormal{Jac}}

\newcommand{\adj}{\textnormal{adj}}

\title{Generalizations of the restriction theorem for multiplier ideals}
\author{Eugene Eisenstein}
\address{Department of Mathematics, University of Michigan, Ann Arbor, MI 48109, USA}
\email{eisenst@umich.edu}
\thanks{The author was partially supported by an NSERC Postgraduate Scholarship.}
\date{\today}

\begin{document}

\begin{abstract} We present an algebro-geometric perspective on some generalizations, due to S. Takagi, of the restriction theorem for multiplier ideals. The first version of the restriction theorem for multiplier ideals was discovered by Esnault and Viehweg, see \cite{EsnaultViehwegVanishing}. In a series of papers (see \cite{TakagiAdjoint} and \cite{TakagiSubadditivity}) S. Takagi has discovered generalizations of the restriction theorem and some formulas for multiplier ideals that follow from the restriction theorem. He uses the technique of tight closure and reduction to positive characteristic. We are able to provide an algebro-geometric proof of generalizations of his restriction theorem in \cite{TakagiAdjoint} and his subadditivity theorem in \cite{TakagiSubadditivity}. We also prove an adjunction formula for relative canonical divisors of factorizing resolutions of singularities.
\end{abstract}

\maketitle

\section{Introduction}                                        

Let $A$ be a smooth complex projective variety and let $Z$ be an $\R_{>0}$-linear combination of subschemes of $A$. In this situation, one can construct a multiplier ideal
\[\cJ(A,Z) \subseteq \cO_A\]
that measures the singularities of the equations defining $Z$. These ideals have found numerous important applications in high-dimensional algebraic geometry. Spectacular applications of the analytic perspective on these ideals have been discovered by Y.-T. Siu in the remarkable papers \cite{SiuInvariance}, \cite{SiuTwistedInvariance}, see also \cite{PaunSiusTheorem}. The celebrated recent work of Hacon and McKernan in \cite{HMFlips}, \cite{HMRC} and \cite{HMBounded}, as well as the work of Takayama in \cite{TakayamaBounded}, uses and expands the concept of the multiplier ideal. There are many more examples of important applications and the theory is still rapidly developing. Chapters 9 through 11 of the excellent book \cite{LazarsfeldPositivityII} provide a good introduction to the subject.

An important tool in the study and application of multiplier ideals is the restriction theorem: if $X \subseteq A$ is a smooth hypersurface not contained in the support of $Z$ then there is an ideal, called the adjoint ideal $\adj_X(A,Z)$, that fits into a short exact sequence
\[0 \to \cJ(A, Z) \otimes \cO_A(-X) \to \adj_X(A,Z) \to \cJ(X,Z_{|X}) \to 0\]
with the right-hand map given by restriction to $X$. In particular, it follows easily from the construction that $\adj_X(A,Z) \subseteq \cJ(A, Z)$ and so we have that $\cJ(X, Z_{|X}) \subseteq \cJ(A,Z) \cdot \cO_X$, the restriction theorem of \cite{EsnaultViehwegVanishing}. This theorem leads to a number of important consequences, such as the subadditivity theorem in \cite{DELSubadditivity}.

The multiplier ideal can be defined under the assumption that $X$ is only $\Q$-Gorenstein. It is natural to ask: to what extent do the above theorems extend to this more general situation? In a series of papers, S. Takagi investigated this question. In \cite{TakagiAdjoint} he defined an adjoint ideal along a  subvariety $X$ of a smooth variety $A$, which we will call the Takagi adjoint ideal, and proved using positive characteristic tight closure techniques that when $X$ is Gorenstein and Cohen-Macaulay then
\[\adj_X(A,Z) \cdot \cO_X = \cJ(X, \V(J_1) + Z_{|X})\]
where $J_1$ is the l.c.i.-defect sheaf of $X$ (see Section \ref{secadjunction} for a brief review of, and references for, l.c.i.-defect sheaves). Our main result is a proof of a stronger statement. We develop a set of techniques that use only standard characteristic zero algebraic geometry and greatly simplify the proof of Takagi's theorem. Our main result is:

\begin{thma}[Theorem \ref{Takagirestriction}] Let $X \subseteq A$ be a $\Q$-Gorenstein subvariety of an ambient smooth variety with a Gorenstein index $r$ and codimension $c$. Let $Z$ be an $\R_{>0}$-linear combination of subschemes of $A$, not containing any component of $X$ in their support. Then there exists a short exact sequence
\[0 \to \cJ(A, cX + Z) \to \adj_X(A,Z) \to \cJ\left(X, \frac{1}{r}\V(J_r) + Z_{|X}\right) \to 0\]
with the right-hand map given by restriction to $X$. In particular,
\[\cJ\left(X, \frac{1}{r}\V(J_r) + Z_{|X}\right) = \adj_X(A,Z) \cdot \cO_X.\]
Furthermore, the Takagi adjoint ideal satisfies the obvious analog of local vanishing.
\end{thma}

In \cite{TakagiSubadditivity}, S. Takagi investigates the subadditivity theorem on a $\Q$-Gorenstein variety $X$. He shows that
\[\Jac_X \cdot \cJ(X, Z_1 + Z_2) \subseteq \cJ(X, Z_1) \cJ(X, Z_2)\]
where $\Jac_X$ is the Jacobian ideal of $X$. To prove this theorem he again uses positive characteristic tight closure methods. We are able to follow an approach similar to our characteristic zero approach to Takagi's restriction theorem to prove his subadditivity theorem without the use of tight closure.

The key idea in our proofs is the systematic use of factorizing resolutions, proven to exist in \cite{BravoVillamayorSFR}. The definition of these resolutions is given in the next section. The importance of factorizing resolutions in our methods stems from the following new variant of the adjunction formula for factorizing resolutions, see Theorem \ref{adjunction}:
\[K_{\overline{X}/X} - \frac{1}{r} D = (K_{\overline{A}/A} - cR_X)_{|\overline{X}}\]
where $A$ is a smooth variety that contains a $\Q$-Gorenstein subvariety $X$, $\pi : \overline{A} \to A$ is a factorizing resolution of $X \subseteq A$ that is also a log-resolution of $J_r$, $\overline{X}$ is the strict transform of $X$, $R_X$ is defined by
\[\cI_X \cdot \cO_{\overline{A}} = \cI_{\overline{X}} \cdot \cO_{\overline{A}}(-R_X),\]
$D$ is given by
\[J_r \cdot \cO_{\overline{X}} = \cO_{\overline{X}}(-D)\]
and $r$ is a Gorenstein index of $X$.

We are indebted to K. Smith for her insight and help on many occasions as well as L. Ein, V. Lozovanu, K. Schwede, S. Takagi, K. Tucker and C. Zeager for useful conversations. Finally, we are very grateful for the help of R. Lazarsfeld and M. Musta\c t\u a. Without their infinite patience and deep insight this paper would not have even been started.

\section{Conventions}

We fix the following notation.

\begin{enumerate}
\item The Kawamata-Viehweg vanishing theorem and resolutions of singularities feature prominently in our methods. We therefore work over an algebraically closed field $k$ of characteristic zero. The case that most concerns us is $k = \C$ and we will assume this for simplicity from now on. A variety is an integral separated scheme of finite type over $k$. We may use the terms \emph{reducible} variety to denote a reduced separated scheme of finite type over $k$.

\item A simple normal crossings variety is a possibly reducible variety $X$, with smooth irreducible components, so that locally analytically at every point of $X$ there exists an isomorphism of $X$ with a subvariety of $\mathbb{A}^n_\C$ defined by unions of intersections of coordinate hyperplanes. A scheme $X$ has simple normal crossings support if $X_{\textnormal{red}}$ is a simple normal crossings variety. We say that $X$ has simple normal crossings with $Y$ if $X \cup Y$ has simple normal crossings support.

In particular, if $X$ is a subscheme of a smooth variety $A$ then $X$ has simple normal crossings support if locally at every point $p \in A$ there exist regular parameters $x_i$ so that the germ at $p$ of the ideal sheaf of $X$ is generated by elements of the form $x_{i_1}^{e_1} \cdots x_{i_s}^{e_s}$.

\item If $\pi : X' \to X$ is a birational morphism then we write $\textnormal{exc}(\pi)$ for the set of points of $X'$ at which $\pi$ is not an isomorphism, endowed with the reduced scheme structure.

\item An embedded resolution of singularities of a generically smooth subscheme $X$ contained in a possibly singular variety $A$ is a birational morphism $\pi : A' \to A$ so that:
\begin{enumerate}
\item $A'$ is smooth and $\pi$ is an isomorphism at every generic point of $X$.
\item The set $\textnormal{exc}(\pi)$ is a divisor with simple normal crossings support.
\item The strict transform of $X$ in $A'$, denoted $X'$, is smooth and has simple normal crossings with $\textnormal{exc}(\pi)$.
\end{enumerate}
Such a resolution exists whenever $X \not\subseteq A_\textnormal{sing}$.

A factorizing resolution of singularities of $X \subseteq A$ as above is a birational morphism $\pi : A' \to A$ that is an embedded resolution of singularities of $X$ in $A$ so that, if $X'$ is the strict transform of $X$ in $A'$, we have that
\[\cI_X \cdot \cO_{A'} = \cI_{X'} \cdot \cL\]
with $\cL$ a line bundle and the support of $\cI_X \cdot \cO_{A'}$ is a simple normal crossings variety. If $A$ is smooth these resolutions were shown to exist in \cite{BravoVillamayorSFR}. We will show in Lemma \ref{sfr} that the case of $A$ singular and $X \not\subseteq A_{\textnormal{sing}}$ follows formally from the smooth case.

Let $Z$ be an $\R_{> 0}$-linear combination of subschemes of $A$ with no component of $X$ contained in the support of $Z$. An embedded resolution of singularities $\pi : A' \to A$ as above is also a log-resolution of $Z$ if $\pi^{-1} Z$ is a divisor with simple normal crossings support and $\Supp(\pi^{-1} Z) \cup \textnormal{exc}(\pi) \cup X'$ is a simple normal crossings variety.

\item $X$ is said to be $\Q$-Gorenstein if $X$ is normal and there is some natural number $r$ so that $rK_X$ is a Cartier divisor. Any such $r$ is called a Gorenstein index of $X$.

\item The abbreviation l.c.i. stands for locally complete intersection. We say that a variety $X$ is l.c.i. at a point $p \in X$ if the local ring $\cO_{X,p}$ is a locally complete intersection ring.

\item If $X \subseteq A$ is an equidimensional subscheme of a variety $A$ we write $\codim_A(X)$ for the codimension of $X$ in $A$.

\item If $\cL$ is a line bundle and $\cF$ is a subsheaf of $\cL$ then we can write $\cF = \cI \cdot \cL$ for some ideal sheaf $\cI$. We will say that $\cF$ generates the ideal $\cI$.

\item If $\cI$ is an ideal sheaf, we denote by $\V(\cI)$ the subscheme defined by $\cI$.

\item A multi-index of type
\[\binom{n}{m}\]
is an ordered list of integers $(i_1, \ldots, i_m)$ so that $i_s < i_{s+1}$ and $i_s \in [1, n]$ for all $s$. If $I$ is a multi-index we write
\[dx_I = dx_{i_1} \wedge \cdots \wedge dx_{i_m}\]
as short-hand for differential forms.

\item If $D$ is a $\Q$-divisor on $A$ and $X$ is a subvariety not contained in the support of $D$ we will write $D_{|X}$ for the intersection of $D$ with $X$ as a $\Q$-divisor on $X$. If $\cF$ is a sheaf on $A$ we will write $\cF_{|X}$ for $\cF \otimes \cO_X$.
\end{enumerate}

\section{A different definition of the Takagi adjoint ideal}

In this section we give a new way to compute the adjoint ideal defined by Takagi in \cite{TakagiAdjoint}. This new way is analogous to the more familiar definition from \cite{LazarsfeldPositivityII}. We will first need the following lemma.

\begin{lemma}\label{sfr} Let $X \subseteq A$ be a generically smooth subscheme of a not necessarily smooth variety $A$. Let $\pi_1 : A' \to A$ be a birational morphism from a smooth variety $A'$ that is an isomorphism at the generic points of the components of $X$. Let $X'$ be the strict transform of $X$ in $A'$ and let $E$ be a divisor on $A'$ with simple normal crossing support so that no component of $X'$ is contained in $E$. Then there exist morphisms
\[\xymatrix{
\overline{X} \ar@{^{(}->}[r]\ar[d] & \overline{A}\ar[d]^{\pi_2} \\
X' \ar@{^{(}->}[r]\ar[d] & A'\ar[d]^{\pi_1} \\
X \ar@{^{(}->}[r] & A
}\]
where $\overline{X}$ is the strict transform of $X$ in $\overline{A}$, which we assume exists, so that $\pi := \pi_1 \circ \pi_2$ is a factorizing resolution of $X$ inside $A$ and $\overline{X} \cup \textnormal{exc}(\pi) \cup \Supp(\pi_2^*E)$ is a simple normal crossings variety.
\end{lemma}
\begin{proof} We perform the following procedure. Take a factorizing resolution of $(\pi_1^{-1}(X))_\textnormal{red}$. Note that the strict transforms of all irreducible components of $(\pi_1^{-1}(X))_\textnormal{red}$ are smooth and disjoint. Blow up the supports of the strict transforms of all irreducible components of $(\pi_1^{-1}(X))_\textnormal{red}$ other than the strict transforms of the components of $X$. Let $\pi^\circ : A'' \to A'$ be the resoluting morphism, let $\pi'' : A'' \to A$ be its composition with $\pi_1$ and consider the subscheme $X''$ of $A''$ defined by the ideal sheaf
\[\cI_{X''} := (\cI_{X} \cdot \cO_{A''}) \cdot \cO_{A''}(-(\pi^\circ)^*E).\]
This is a scheme supported on the strict transform of $X$ and a union of divisors on a smooth variety but it may have some embedded primes.

Since $A''$ is smooth the divisorial components of $X''$ are locally principal. The embedded primes of $X''$ are supported either on the strict transform of $X$, which is generically reduced, or on one of these divisorial components. Write the divisorial part of $X''$ as
\[\sum a_i D_i\]
with $a_i > 0$ and let
\[\cL = \cO_{A''} \left( \sum (a_i - 1) D_i \right).\]
The subscheme $Y$ of $A''$ defined by the ideal $\cI_Y = \cI_{X''} \cdot \cL$ is generically reduced. We conclude by taking a factorizing resolution $\overline{A}$ of $Y$, which is now possible since it is generically smooth, and noticing that the expansion of $\cI_Y$ and $\cI_{X''}$ to this resolution must differ by the pull-back of $\cL$. By definition of $X''$ it follows that
\[\cI_X \cdot \cO_{\overline{A}} = \cI_{\overline{X}} \cdot \cM\]
for a line bundle $\cM$ and the simple normal crossings hypothesis in the lemma is also satisfied.
\end{proof}

\begin{cor}\label{sfrcor} If no component of $X$ is contained in $A_\textnormal{sing}$ then a factorizing resolution of $X$ always exists. Furthermore we can choose this resolution to be a log-resolution $\pi$ of any $\R_{>0}$-linear combination $Z$ of subschemes of $A$ not containing any component of $X$ in its support.
\end{cor}
\begin{proof} Let $\pi' : A' \to A$ be any birational morphism with $A'$ smooth that is an isomorphism at the generic points of $X$. Take a log-resolution $\pi'' : A'' \to A'$ of $(\pi')^{-1}Z + \textnormal{exc}(\pi)$. Let $\pi_1 = \pi' \circ \pi''$ and let
\[E = (\pi'')^{-1}(\Supp(\pi')^{-1}Z + \textnormal{exc}(\pi)).\]
We apply the previous lemma to this $\pi_1$ and $E$ to conclude.
\end{proof}

We are now ready to discuss the adjoint ideal. First we note that Takagi's original definition generalizes naturally to our setting.

\begin{defn}[Takagi]\label{takadj} Let $X$ be a generically smooth subscheme of a possibly singular variety $A$. Let $c = \codim_A(X)$ and let $Z$ be any $\R_{>0}$-linear combination of subschemes of $A$ with no component of $X$ contained in $\Supp(Z) \cup A_\textnormal{sing}$. Finally suppose that $A$ is $\Q$-Gorenstein. Let $f : A' \to A$ be the blow-up of $X$ and let $E$ be the reduced but possibly reducible exceptional divisor of $f$ that dominates $X$. Let $g : A'' \to A'$ be a log-resolution of $f^{-1}(X) + f^{-1}(Z)$ so that $g_*^{-1}E$ is smooth and let $\pi = f \circ g$. Set
\[\adj_X(A,Z) := \pi_*\cO_{A''}(\lceil K_{A''/A} - \pi^{-1}Z - c \cdot \pi^{-1}X + g_*^{-1}E \rceil).\]
\end{defn}

It is worthwhile to remark that the proof of Lemma 1.6 in \cite{TakagiAdjoint} goes through \emph{mutatis mutandis}, so that Takagi's definition is independent of the choice of log-resolution $g$, even if $A$ is singular. We can now state the definition we propose for the Takagi adjoint ideal.

\begin{defn}\label{myadj} Let $X$ be a generically smooth equidimensional subscheme of a variety $A$. Let $c = \codim_A(X)$ and let $Z$ be any $\R_{>0}$-linear combination of subschemes of $A$ with no component of $X$ contained in $\Supp(Z) \cup A_\textnormal{sing}$. Finally suppose that $A$ is $\Q$-Gorenstein. Let $\pi : \overline{A} \to A$ be a log-resolution of $Z$ that is also a factorizing resolution of $X$ in the sense of Corollary \ref{sfrcor}. Let $\overline{X}$ be the strict transform of $X$ in $\overline{A}$. Write
\[\cI_X \cdot \cO_{\overline{A}} = \cI_{\overline{X}} \cdot \cO_{\overline{A}}(-R_X).\]
We define
\[\adj'_X(A, Z) := \pi_* \cO_{\overline{A}}(\lceil K_{\overline{A}/A} - \pi^{-1}(Z) \rceil - cR_X).\]
\end{defn}

First we prove that these two definitions always compute the same ideal.

\begin{prop} Notation as in the preceding definition. Then
\[\adj_X(A, Z) = \adj'_X(A, Z).\]
\end{prop}
\begin{proof} Let $\pi$ be a factorizing resolution as in Definition \ref{myadj} and let $A''$ be the blow-up of $\overline{A}$ along $\overline{X}$. Let $\pi'' : A'' \to \overline{A}$ be the blow-up morphism. Let $\pi'$ be the composition $A'' \to A$. Notice that due to the simple normal crossings hypotheses the composition $A'' \to A$ satisfies the conditions of Definition \ref{takadj}. Let $E$ be the (reduced) union of the exceptional divisors lying above the generic points of the irreducible components of $X$. Since $\pi''$ is a blow-up of smooth centers transverse to the exceptional locus of $\pi$ we compute:
\begin{align*}
(\pi'')^* \cO_{\overline{A}}(\lceil K_{\overline{A}/A} - \pi^{-1} Z - c R_X \rceil)&= \cO_{A''}(\lceil K_{A''/A} - (\pi')^{-1} Z - c R_X \rceil - (c-1) E)\\
&= \cO_{A''}(\lceil K_{A''/A} - (\pi')^{-1} Z  - c(\pi')^{-1}(X) + E \rceil).
\end{align*}
By the universal property of blow-ups, $\pi'$ must factor through the blow-up of $X$ in $A$. But then the divisor we just arrived at computes Takagi's adjoint. We conclude by the projection formula.
\end{proof}

From now on we will conflate the two notations, that is, we will write $\adj'_X(A, Z)$ as $\adj_X(A, Z)$. Note that we already noticed that Takagi has already shown that $\adj_X(A, Z)$ does not depend on the choice of resolution and so our adjoint does not either.

At this stage we will also prove a formula that may seem technical at first but packages the application of the local vanishing theorem that we will use to deduce our restriction theorems.

\begin{lemma}\label{localvanishing} Notation as in Definition \ref{myadj}. Let $D := \lceil K_{\overline{A}/A} - \pi^{-1}(Z) \rceil - cR_X$. We have the usual short exact sequence
\begin{equation}\label{ses}
0 \to \cI_{\overline{X}} \cdot \cO_{\overline{A}}(D) \to \cO_{\overline{A}}(D) \to \cO_{\overline{X}}(D_{|\overline{X}}) \to 0.
\end{equation}
Then
\[R^i\pi_*(\cI_{\overline{X}} \cdot \cO_{\overline{A}}(D)) = 0\]
for all $i > 0$. In particular, if $f$ is the restriction of $\pi$ to $\overline{X}$,
\[\adj_X(A, Z) \cdot \cO_X = f_*(\cO_{\overline{A}}(\lceil K_{\overline{A}/A} - \pi^{-1}(Z) \rceil - cR_X) \cdot \cO_{\overline{X}}).\]
In other words we may restrict first then push forward. Furthermore,
\[\pi_* \cI_{\overline{X}} \cdot \cO_{\overline{A}}(D) = \cJ(A, cX + Z).\]
\end{lemma}
\begin{proof} We calculate as follows. Let $\pi'' : A'' \to \overline{A}$ be the blow-up of $\overline{X}$ with reduced exceptional divisor $E$. Let $\pi'$ be the composition $A'' \to A$. Then
\begin{align*}
(\cI_{\overline{X}} \cdot \cO_{A''}) \cdot (\pi'')^* \cO_{\overline{A}}(\lceil K_{\overline{A}/A} - \pi^{-1} Z \rceil - cR_X)&= \cO_{A''}((\pi'')^*(\lceil K_{\overline{A}/A} - \pi^{-1} Z \rceil - cR_X) - E)\\
&= \cO_{A''}(\lceil K_{A''/A} - (\pi')^{-1} Z - c (\pi')^{-1} X\rceil).
\end{align*}
This has vanishing higher direct images by the local vanishing theorem, see \cite{LazarsfeldPositivityII}, Theorem 9.4.1. Furthermore,
\[\pi_* \cI_{\overline{X}} \cdot \cO_{\overline{A}}(D) = \pi'_*\cO_{A''}(\lceil K_{A''/A} - (\pi')^{-1} Z - c (\pi')^{-1} X\rceil) = \cJ(A, cX + Z).\]
In turn, local vanishing for $\pi''$ implies that
\[R^i\pi''_* (\cO_{A''}(\lceil K_{A''/A} - (\pi')^{-1} Z - c (\pi')^{-1} X\rceil)) = 0\]
for all $i > 0$. We conclude by the following lemma.
\end{proof}

\begin{lemma} Suppose we are given a diagram of proper morphisms
\[\xymatrix{
& X'' \ar[dd]^{f}\ar[dl]_{h}\\
X' \ar[dr]_{g} &\\
& X
}\]
and a coherent sheaf $\cF$ on $X''$. Suppose that $R^j h_* \cF$ and $R^j f_* \cF$ vanish for all $j > 0$. Then $R^j g_* (h_* \cF) = 0$ for all $j > 0$.
\end{lemma}
\begin{proof} This is an easy consequence of the Leray spectral sequence. Indeed, the Leray spectral sequence has the form
\[R^i g_* (R^j h_* \cF) \Rightarrow R^{i+j} f_* \cF.\]
Since $R^j h_* \cF = 0$ for all $j > 0$ the spectral sequence degenerates at the $E_2$ sheet. Now the assumption that $R^j f_* \cF = 0$ for all $j > 0$ immediately implies the conclusion.
\end{proof}

\section{Jacobian ideals and the adjunction formula in high codimension}\label{secadjunction}

We will first briefly review the relevant notions of lci-defect sheaves. For details see the appendix of \cite{EinMustataJetSchemes}.

Suppose that $X$ is $\Q$-Gorenstein with a Gorenstein index $r$. We get a map
\[(\Omega_X^n)^{\otimes r} \to \cO_X(rK_X)\]
giving by restricting a section on the left to $X \setminus X_\textnormal{sing}$ and extending to a section of $\cO_X(rK_X)$, which is possible since $X$ is normal. Let $\cI_{r,X}$ be the ideal generated by the image of this map. Then $\Jac_X^r \subseteq \cI_{r, X}$ with equality if and only if $X$ is locally a complete intersection. In general there is an ideal sheaf $J_r$ so that
\[\overline{J_r \cdot \cI_{r, X}} = \overline{\Jac_X^r}.\]
This ideal sheaf is called the $r$-th lci-defect sheaf of $X$. It is worthwhile at this point to note that $(J_r)^s \subseteq J_{rs}$ and $\overline{(J_r)^s} = \overline{J_{rs}}$.
Note also that we have adopted the notation of \cite{EinMustataJetSchemes} to emphasize the dependence of the lci-defect sheaves on the Gorenstein index.

We now discuss the adjunction formula that we will use in our proof of Takagi's restriction theorem. We begin with a formula for Jacobian ideals. First we introduce some notation.

\begin{defn} Let $A$ be a matrix of elements of a ring $R$. We denote by $[A]_n$ the ideal of $R$ generated by the $n \times n$-minors of $A$.
\end{defn}

Let $f : Y \to X$ be a morphism of possibly reducible varieties of pure dimension with $Y$ smooth and $\dim(X) = \dim(Y) = n$. Consider the natural map $f^*\Omega_X^n \to \Omega_Y^n$. The image of this map is, by definition, given by $\Jac_f \cdot \Omega_Y^n$. If $X$ is also smooth then, in local coordinates, $\Jac_f$ is just $[df]_n$. We first prove a general lemma regarding Jacobian ideals that can be viewed as a kind of chain rule. It will be useful in the current generality in our investigation of the subadditivity theorem. First, we make a few definitions.

\begin{setup}\label{jacsetup}\ 
\begin{enumerate}
\item[(a)] Denote by $A$ a smooth variety of dimension $N$ and $X$ an equidimensional possibly reducible subvariety of dimension $n$ and codimension $c$. Set $\fa$ to be an ideal sheaf on $A$ contained in the ideal sheaf of $X$. We denote by $\pi : A' \to A$ a birational morphism with $A'$ smooth that is furthermore an isomorphism at every generic point of $X$. Denote by $X'$ the strict transform of $X$ along $\pi$. Suppose that $X'$ is smooth. Let $f : X' \to X$ be the restriction of $\pi$.

\item[(b)] Let $p \in A'$ and let the germ of $\fa$ at $\pi(p)$ be generated by $(h_1, \ldots, h_m)$. Let $w_1, \ldots, w_N$ be local coordinates of $A'$ at $p$ and let $z_1, \ldots, z_N$ be local coordinates of $A$ at $\pi(p)$. We suppose finally that $w_1, \ldots, w_n$ restrict to local coordinates on $X'$ at $p$ and that all other $w_j$ restrict to zero on $X'$.

\item[(c)] To distinguish the two constructions, in the case where $\pi$ is a factorizing resolution for $X$ we will write $\overline{A}$ instead of $A'$ and $\overline{X}$ instead of $X'$.
\end{enumerate}
\end{setup}

With these choices we have the following formula.

\begin{lemma}\label{changeofcoords} Notation as in Setup \ref{jacsetup}. As germs at $p$,
\[\Jac_f \cdot \left(\left[ \frac{\partial (h_i \circ \pi)}{\partial w_j} \right]_c\right) \cdot \cO_{X'} = \left(\Jac_\pi \cdot \left[ \frac{\partial h_i}{\partial z_j}\right]_c \cdot \cO_{A'}\right) \cdot \cO_{X'}.\]
Here $\pi$ need not be factorizing for $X$.
\end{lemma}
\begin{proof} If $m < c$ the statement is trivial so let $I$ be a multi-index of type $\binom{N}{n}$ and let $J$ be a multi-index of type $\binom{m}{c}$. Consider the form
\[\omega_{I,J} = d(z_{i_1} \circ \pi) \wedge \cdots \wedge d(z_{i_n} \circ \pi) \wedge d(h_{j_1} \circ \pi) \wedge \cdots \wedge d(h_{j_c} \circ \pi).\]
The form $\omega_{I,J}$ is an element of the module $(\Omega_{A'}^N)_p$. Let $\fb$ be the ideal generated by the $\omega_{I,J}$ for all choices of $I$ and $J$. On the one hand,
\begin{align*}
\omega_{I,J}&= \pi^* (dz_{i_1} \wedge \cdots \wedge dz_{i_n} \wedge dh_{j_1} \wedge \cdots \wedge dh_{j_c})\\
&= \pm \pi^*\left( m_{I^c, J} \cdot (dz_1 \wedge \cdots \wedge dz_N)\right)\\
&= \pm \Jac_\pi \cdot (m_{I^c, J} \circ \pi) \cdot (dw_1 \wedge \cdots \wedge dw_N).
\end{align*}
where $m_{I^c, J}$ is the minor of the matrix of partials $\frac{\partial h_i}{\partial z_j}$ corresponding to the rows $(1, \ldots, N) \setminus I$ and columns $J$. It follows that
\[\fb \cdot \cO_{X'} = \left(\Jac_\pi \cdot \left[ \frac{\partial h_i}{\partial z_j}\right]_c \cdot \cO_{A'} \right) \cdot \cO_{X'}.\]

Now, observe that, for any $i$ we have
\[d(h_i \circ \pi)_{|X'} = d((h_i \circ \pi)_{|X'}) = 0\]
since the $h_i$ vanish on $X$. On the other hand,
\[d(h_i \circ \pi)_{|X'} = \sum_{j = 1}^N \left(\frac{\partial (h_i \circ \pi)}{\partial w_j}\right)_{|X'} d(w_{j|X'}).\]
We chose $w_{1|X'}, \ldots, w_{n|X'}$ to be local coordinates on $X'$ at $p$, so the $d(w_{j|X'})$ are linearly independent for $1 \leq j \leq n$ while the rest are zero. It follows that we must have
\[\left(\frac{\partial (h_i \circ \pi)}{\partial w_j}\right)_{|X'} = 0\]
for $1 \leq j \leq n$.

Now, if we define $\pi_i = z_i \circ \pi$, then
\begin{align*}
\omega_{I,J} = &\left( \sum_{S \text{ type } \binom{N}{n}} \left(\frac{\partial\pi_{i_1}}{\partial w_{s_1}} \cdots \frac{\partial\pi_{i_n}}{\partial w_{s_n}}\right) \cdot dw_{s_1} \wedge \cdots \wedge dw_{s_n}\right) \wedge\\
&\left( \sum_{T \text{ type } \binom{N}{c}} \left(\frac{\partial (h_{j_1} \circ \pi)}{\partial w_{t_1}} \cdots \frac{\partial (h_{j_c} \circ \pi)}{\partial w_{t_c}}\right) \cdot dw_{t_1} \wedge \cdots \wedge dw_{t_c}\right).
\end{align*}
By our calculation of the derivatives of $h_i \circ \pi$, the terms
\[\left(\frac{\partial (h_{j_1} \circ \pi)}{\partial w_{t_1}} \cdots \frac{\partial (h_{j_c} \circ \pi)}{\partial w_{t_c}}\right)_{|X'}\]
are non-zero only if $T = (n+1, \ldots, N)$. It follows that
\[\fb \cdot \cO_{X'} = \left([d\pi]_{n, (1, \ldots, n)} \cdot \left[ \frac{\partial (h_i \circ \pi)}{\partial w_j} \right]_c\right) \cdot \cO_{X'}\]
where $[d\pi]_{n, (1, \ldots, n)}$ is the ideal of $n \times n$-minors of $d\pi$ with the choice of columns (here the columns give the variables that we differentiate with respect to) equal to $(1, \ldots, n)$. But, since $w_1, \ldots, w_n$ were chosen to restrict to the local coordinates of $X'$ and $w_{n+1}, \ldots, w_N$ were chosen to restrict to zero on $X'$ it is immediate that
\[([d\pi]_{n, (1, \ldots, n)}) \cdot \cO_{X'} = \Jac_f.\]
\end{proof}

The following lemma is essentially the adjunction formula that we will use to deduce Takagi's restriction theorem.

\begin{lemma}\label{jacadj} Notation as in (\ref{jacsetup}). Suppose furthermore that $\pi$ is a factorizing resolution\footnote{Recall that, in Setup \ref{jacsetup}, we agreed to use $\overline{A}$ and $\overline{X}$ in the notation for factorizing resolutions.} of $X$. Write
\[\cI_X \cdot \cO_{\overline{A}} = \cI_{\overline{X}} \cdot \cO_{\overline{A}}(-R_X).\]
Then
\[(\Jac_\pi \cdot \cO_{\overline{X}}) \cdot  (\Jac_X \cdot \cO_{\overline{X}}) = \Jac_f \cdot (\cO_{\overline{A}}(-cR_X)) \cdot \cO_{\overline{X}}.\]
\end{lemma}
\begin{proof} We apply the previous lemma. Choose $\fa$ to be the ideal of $X$. Suppose that, at $p$, $I = (h_1, \ldots, h_m)$. We get that
\[\Jac_f \cdot \left(\left[ \frac{\partial (h_i \circ \pi)}{\partial w_j} \right]_c\right) \cdot \cO_{\overline{X}} = \left(\Jac_\pi \cdot \left[ \frac{\partial h_i}{\partial z_j}\right]_c \cdot \cO_{\overline{A}} \right) \cdot \cO_{\overline{X}}.\]
By definition,
\[\left(\left[ \frac{\partial h_i}{\partial z_j}\right]_c\right) \cdot \cO_X = \Jac_X.\]
Now, as germs at $p$ we may write
\[\cI_X \cdot \cO_{\overline{A}} = (h_1 \circ \pi, \ldots, h_m \circ \pi) = (g\overline{h}_1, \ldots, g\overline{h}_m)\]
with $g$ a local generator of $\cO_{\overline{A}}(-R_X)$. But
\[\frac{\partial (h_i \circ \pi)}{\partial w_j} = \frac{\partial (g\overline{h}_i)}{\partial w_j} = g\frac{\partial (\overline{h}_i)}{\partial w_j} + \overline{h_i}\frac{\partial g}{\partial w_j}.\]
Since $\overline{h}_i$ restricts to zero on $\overline{X}$, we see that
\[\left(\left[ \frac{\partial (h_i \circ \pi)}{\partial w_j} \right]_c\right) \cdot \cO_{\overline{X}} = g^c \cdot \Jac_{\overline{X}} = (g^c)\]
since $\overline{X}$ is smooth.
\end{proof}

The next step is to interpret the various Jacobian ideals that appear in this formula in terms of relative canonical classes.

\begin{lemma}\label{adjunction2} Let $X$ be a $\Q$-Gorenstein possibly reducible variety and let $f : \overline{X} \to X$ be a birational morphism with $\overline{X}$ smooth. Let $r$ be a Gorenstein index of $X$. Then
\[\Jac_f^r = (\cI_{r,X} \cdot \cO_{\overline{X}}) \cdot \cO_{\overline{X}}(-rK_{\overline{X}/X}).\] 
\end{lemma}
\begin{proof} Write
\[K_{\overline{X}} + K^- = f^* K_X + K^+\]
with $K^+$, $K^-$ effective. We get a map
\[f^* \cO_X(rK_X) \otimes \cO_{\overline{X}}(-rK^-) \to \cO_{\overline{X}}(r(K_{\overline{X}})).\]
The image of this map is given by $I \cdot \cO_{\overline{X}}(rK_{\overline{X}})$ where $I$ is the ideal $\cO_{\overline{X}}(-rK^+)$. Next we have a commutative diagram
\[\xymatrix{
f^* (\Omega_X^n)^{\otimes r} \otimes \cO_{\overline{X}}(-rK^-) \ar[r]\ar[dr] & \cO_{\overline{X}}(rK_{\overline{X}}) \\
& f^* \cO_{\overline{X}}(rK_X) \otimes \cO_{\overline{X}}(-rK^-). \ar[u]
}\]
By computing the images of these maps we see that
\[\Jac_f^r \cdot \cO_{\overline{X}}(-rK^-) = (\cI_{r,X} \cdot \cO_{\overline{X}}) \cdot \cO_{\overline{X}}(-rK^+).\]
The required statement follows by rearranging this equation.
\end{proof}

We can now finally prove our adjunction formula.

\begin{thm}\label{adjunction} Let $A$ be a smooth variety and let $X$ be a generically smooth equidimensional subscheme. Let $\pi : \overline{A} \to A$ be a factorizing resolution of $X$ inside $A$ and let $f$ be the restriction of $\pi$ to $\overline{X}$, the strict transform of $X$ along $\pi$. Write
\[\cI_X \cdot \cO_{\overline{A}} = \cI_{\overline{X}} \cdot \cO_{\overline{A}}(-R_X).\]
Suppose that $X$ is $\Q$-Gorenstein with a Gorenstein index $r$. Suppose further that $f$ is a log-resolution of $\cI_{r,X}$ and $J_r$. Let $D$ be the divisor defined by
\[J_r \cdot \cO_{\overline{X}} = \cO_{\overline{X}}(-D).\]
Then
\[K_{\overline{X}/X} - \frac{1}{r}D = (K_{\overline{A}/A} - cR_X)_{|\overline{X}}\]
with equality being equality of $\Q$-divisors on $\overline{X}$.
\end{thm}
\begin{proof} This follows easily from the previous two lemmas and the definition of $J_r$.
\end{proof}

Note that the necessary $\pi$, that is, a log-resolution of $\cI_{r,X}$ and $J_r$ can always be found by Lemma \ref{sfr}. Lastly we record the following easy fact that will be useful for the subadditivity theorem.

\begin{lemma}\label{jaccomposition} Let $f_1 : A \to B$ and $f_2 : B \to C$ be birational morphisms and let $f = f_2 \circ f_1$. Suppose that $A$ and $B$ are smooth. Then
\[\Jac_f = (\Jac_{f_2} \cdot \cO_A) \cdot \Jac_{f_1}\]
\end{lemma}
\begin{proof} Let $n$ be the dimension of the varieties involved. Consider the composition of natural maps
\[f^* \Omega_C^n \to f_2^*\Omega_B^n \to \Omega_A^n.\]
The composition is $df$ and the formula follows easily since $\Omega_B^n$ and $\Omega_A^n$ are line bundles and the second morphism is just multiplication by a generator of $\Jac_{f_1}$.
\end{proof}

\section{Takagi's restriction theorem}

The tools developed so far enable us to give a quick proof of a stronger form of the restriction theorem given by Takagi in his paper \cite{TakagiAdjoint}.

\begin{thm}\label{Takagirestriction} Let $X \subseteq A$ be a $\Q$-Gorenstein (in particular, reduced) equidimensional subscheme of an ambient smooth variety with a Gorenstein index $r$ and codimension $c$. Let $Z$ be an $\R_{>0}$-linear combination of subschemes of $A$, not containing any component of $X$ in their support. Then there exists a short exact sequence
\[0 \to \cJ(A, cX + Z) \to \adj_X(A,Z) \to \cJ\left(X, \frac{1}{r}\V(J_r) + Z_{|X}\right) \to 0\]
with the first map given by inclusion and the last map given by restriction to $X$.
\end{thm}
\begin{proof} Let $\pi : \overline{A} \to A$ be a factorizing resolution as in Definition \ref{myadj} and in Theorem \ref{adjunction}. Lemma \ref{localvanishing} gives the short exact sequence
\[0 \to \cI_{\overline{X}} \cdot \cO_{\overline{A}}(D) \to \cO_{\overline{A}}(D) \to \cO_{\overline{X}}(D_{|\overline{X}}) \to 0\]
where $D := \lceil K_{\overline{A}/A} - \pi^{-1}(Z) \rceil - cR_X$. The same lemma states that
\[R^i\pi_*(\cI_{\overline{X}} \cdot \cO_{\overline{A}}(D)) = 0\]
for all $i > 0$. We have already seen in Lemma \ref{localvanishing} that
\[\pi_* (\cI_{\overline{X}} \cdot \cO_{\overline{A}}(D)) = \cJ(A, cX + Z)\]
and
\[\pi_* \cO_{\overline{A}}(D) = \adj_X(A, Z)\]
by definition. We must only check that
\[f_* \cO_{\overline{X}}(D_{|\overline{X}}) = \cJ\left(X, \frac{1}{r}\V(J_r) + Z_{|X}\right).\]
This follows immediately from the following expresssion
\begin{equation}\label{appliedadj}
\cO_{\overline{X}}(D_{|\overline{X}}) = \cO_{\overline{X}}\left(\left\lceil K_{\overline{X}/X} - \frac{1}{r}\pi^{-1}(J_r) - \pi^{-1}(Z_{|X}) \right\rceil\right)
\end{equation}
which itself follows immediately from the formula of Theorem \ref{adjunction}.
\end{proof}

The following form of our restriction theorem answers a question of Takagi's in \cite{TakagiAdjoint}, Remark 3.2, (3).

\begin{cor} In the situation of the theorem we have the formulas
\[\adj_X(A,Z) \cdot \cO_X = \cJ\left(X, \frac{1}{r}\V(J_r) + Z_{|X}\right),\ \cJ\left(X, \frac{1}{r}\V(J_r) + Z_{|X}\right) \subseteq \cJ(A,Z) \cdot \cO_X.\]
\end{cor}
\begin{proof} The first expression follows immediately from Theorem \ref{Takagirestriction} while the second follows from the easy observation that $\adj_X(A,Z) \subseteq \cJ(A, Z)$ by definition.
\end{proof}

\begin{cor} Notation as in the definition of the adjoint ideal. The adjoint ideal satisfies local vanishing, that is,
\[R^i\pi_* \cO_{\overline{A}}(\lceil K_{\overline{A}/A} - \pi^{-1}(Z) \rceil - cR_X) = 0\]
for all $i > 0$.
\end{cor}
\begin{proof} This follows from the long exact sequence for $R^i\pi_*$ that arises from pushing forward
\[0 \to \cI_{\overline{X}} \cdot \cO_{\overline{A}}(D) \to \cO_{\overline{A}}(D) \to \cO_{\overline{X}}(D_{|\overline{X}}) \to 0\]
in Lemma \ref{localvanishing}. In this lemma we have already seen that the term on the left has vanishing higher direct images. On the other hand, the expression for $\cO_{\overline{X}}(D_{|\overline{X}})$ in (\ref{appliedadj}) shows that this sheaf has vanishing higher direct images by the local vanishing theorem again, see \cite{LazarsfeldPositivityII}, Theorem 9.4.1.
\end{proof}

\section{Takagi's subadditivity theorem}

In \cite{TakagiSubadditivity}, Takagi proves the following theorem.

\begin{thm}\label{subadditivity} Let $X$ be a $\Q$-Gorenstein variety and let $Z_1, Z_2$ be $\R_{>0}$-linear combinations of subschemes of $X$. Then
\[\Jac_X \cdot \cJ(X, Z_1 + Z_2) \subseteq \cJ(X, Z_1) \cdot \cJ(X, Z_2).\]
\end{thm}

This is a generalization of the subadditivity theorem from \cite{DELSubadditivity} to singular varieties. In his paper, Takagi uses tight closure techniques in positive characteristic to reduce the theorem to a problem in commutative algebra. We will present a new proof of a slightly stronger statement that uses only standard algebro-geometric characteristic zero techniques: resolution of singularities and Kawamata-Viehweg vanishing.

We proceed with the proof. Our approach will be similar to the approach of \cite{DELSubadditivity}. Specifically, we will make use of the following observation.

\begin{lemma}\label{robslemma} Let $X_1, X_2$ be $\Q$-Gorenstein varieties and let $Z_1$ and $Z_2$ be $\R_{>0}$-linear combinations of subschemes of $X_1$ and $X_2$, respectively. Let $p_1$ and $p_2$ be the projections from $X_1 \times X_2$ to $X_1$ and $X_2$ respectively. Then
\[\cJ(X_1 \times X_2, p_1^*Z_1 + p_2^* Z_2) = p_1^{-1}\cJ(X_1, Z_1) \cdot p_2^{-1}\cJ(X_2, Z_2).\]
\end{lemma}
\begin{proof} The proof in \cite{LazarsfeldPositivityII}, Proposition 9.5.22, can be easily modified to remove the requirement that $X_1$ and $X_2$ be smooth.
\end{proof}

From now on let $X$ be a $\Q$-Gorenstein variety of dimension $n$ and a Gorenstein index $r$ and let $\Delta \subseteq X \times X$ be the diagonal. Let $g : X' \to X$ be a proper birational morphism from a smooth variety $X'$ and let $\rho : X' \times X' \to X \times X$ be the product morphism. Let $\Delta'$ be the strict transform of $\Delta$ in $X' \times X'$. Notice that $\Delta'$ is the diagonal of $X' \times X'$ and the induced morphism $\Delta' \to \Delta$ is just $g$.

The obstruction to the proof of Theorem \ref{subadditivity} is the restriction theorem: it requires a smooth ambient space. We will show that, in our very special situation of the diagonal in $X \times X$, an appropriate restriction theorem holds. To this end we need to uncover the analog of the adjunction formula. The following lemma becomes precisely the required analog once we use our earlier work to translate the Jacobian ideals into relative canonical classes.

\begin{lemma}\label{prodjac} Let
\[\xymatrix{
\overline{\Delta} \ar@{^{(}->}[r]\ar[d]_{f} \ar@/_2.2pc/[dd]_h & A' \ar[d]^{\pi} \ar@/^2.2pc/[dd]^\sigma \\
\Delta' \ar@{^{(}->}[r]\ar[d]_{g} & X' \times X' \ar[d]^{\rho} \\
\Delta \ar@{^{(}->}[r] & X \times X
}\]
be a factorizing resolution of $\Delta$ given by Lemma \ref{sfr}. Let
\[\cI_\Delta \cdot \cO_{A'} = \cI_{\overline{\Delta}} \cdot \cO_{A'}(-R_\Delta).\]
Then
\[(\Jac_{\pi} \cdot \cO_{\overline{\Delta}}) \cdot (\Jac_{g} \cdot \cO_{\overline{\Delta}})^2 \subseteq \Jac_h \cdot (\cO_{A'}(-nR_\Delta)) \cdot \cO_{\overline{\Delta}}.\]
\end{lemma}
\begin{proof} To simplify and unify notation, let $A = X' \times X'$ and $B = X \times X$ and fix a point $p \in A'$. Let $w_j$ be coordinates on $A'$ and let $z_j$ be coordinates on $A$ at $\pi(p)$ as in Setup \ref{jacsetup}. Let $\text{pr}_1, \text{pr}_2 : X \times X \to X$ be the two projections. Let $s_{1,j}$ be generators on $X$ of the maximal ideal of the local ring at $\text{pr}_1(\sigma(p))$, where $1 \leq j \leq M$ and similarly let $s_{2,j}$ be generators on $X$ of the maximal ideal of the local ring at $\text{pr}_2(\sigma(p))$. Finally, let $x_i = \text{pr}_1^* (s_{1,i})$ and $y_i = \text{pr}_2^* (s_{2,i})$. Then
\[(\cI_\Delta)_{\sigma(p)} = (g_1, \ldots, g_M)\]
where $g_i = x_i - y_i$. By Lemma \ref{changeofcoords} applied with the above choices, $h_i = g_i \circ \rho$, $X = \Delta'$ and $X' = \overline{\Delta}$ we have that
\begin{equation}\label{jac1}
\Jac_f \cdot \left( \left[ \frac{\partial (g_i \circ \sigma)}{\partial w_j} \right]_n \right) \cdot \cO_{\overline{\Delta}} = \left( \Jac_{\pi} \cdot \left[ \frac{\partial(g_i \circ \rho)}{\partial z_j} \right]_n \cdot \cO_{A'}\right) \cdot \cO_{\overline{\Delta}}.
\end{equation}

Suppose that we can show that on $\Delta'$ we have
\begin{equation}\label{jac2}
(\Jac_{g})^2 \subseteq \Jac_{g} \cdot \left( \left[ \frac{\partial (g_i \circ \rho)}{\partial z_j} \right]_n \cdot \cO_{\Delta'} \right).
\end{equation}
Then, after multiplying both sides of (\ref{jac1}) by $\Jac_g$ and using Lemma \ref{jaccomposition}, we obtain first of all that
\[(\Jac_{\pi} \cdot \cO_{\overline{\Delta}}) \cdot (\Jac_{g} \cdot \cO_{\overline{\Delta}})^2 \subseteq \Jac_h \cdot \left( \left[ \frac{\partial (g_i \circ \sigma)}{\partial w_j} \right]_n \cdot \cO_{\overline{\Delta}} \right).\]
As in the proof of Lemma \ref{jacadj}, write next
\[(\cI_\Delta \cdot \cO_{A'})_p = (g_1 \circ \sigma, \ldots, g_M \circ \sigma) = (r \overline{g}_1, \ldots, r \overline{g}_M)\]
where $r$ is a local generator of the sheaf $\cO_{A'}(-R_{\overline{\Delta}})$. We now have that
\[\left(\frac{\partial (g_i \circ \sigma)}{\partial w_j}\right)_{|\overline{\Delta}} = \left( r \frac{\partial \overline{g}_i}{\partial w_j} + \overline{g}_i \frac{\partial r}{\partial w_j} \right)_{|\overline{\Delta}} = \left( r \frac{\partial \overline{g_i}}{\partial w_j} \right)_{|\overline{\Delta}}\]
and so
\[\left( \left[ \frac{\partial (g_i \circ \sigma)}{\partial w_j} \right]_n \cdot \cO_{\overline{\Delta}} \right) = r^n \cdot \Jac_{\overline{\Delta}} = (r^n)\]
since $\overline{\Delta}$ is smooth. But this concludes the proof, assuming (\ref{jac2}).

It remains to show (\ref{jac2}). In fact, it is clearly enough to show that
\begin{equation}\label{jac4}
\Jac_{g} \subseteq \left[ \frac{\partial (g_i \circ \rho)}{\partial z_j}\right]_n \cdot \cO_{\Delta'}.
\end{equation}
For this we choose the $z_j$ as follows: let $\text{pr}'_1, \text{pr}'_2 : X' \times X' \to X'$ be the two projections and let $s'_{1,1}, \ldots, s'_{1,n}$ be local coordinates on $X'$ at $\text{pr}'_1(\pi(p))$, $s'_{2,1}, \ldots, s'_{2,n}$ local coordinates on $X'$ at $\text{pr}'_2(\pi(p))$. Let $x_i' = (\textnormal{pr}'_1)^*(s'_{1,i})$ and $y_i' = (\textnormal{pr}'_2)^*(s'_{2,i})$. Set $z_j = x'_j$ for $1 \leq j \leq n$ and $z_j = y'_{j-n}$ for $n + 1 \leq j \leq 2n$.

Notice that, since $y_i \circ \rho$ does not depend on $x'_j$ we have that
\[\frac{\partial (y_i \circ \rho)}{\partial x'_j} = 0.\]
It follows that, in these coordinates, the matrix of partials
\[\frac{\partial (g_i \circ \rho)}{\partial z_j} = \frac{\partial (x_i \circ \rho)}{\partial z_j} - \frac{\partial (y_i \circ \rho)}{\partial z_j}\]
is block diagonal with two blocks,
\[\frac{\partial (x_i \circ \rho)}{\partial x'_j},\ -\frac{\partial (y_i \circ \rho)}{\partial y'_j}.\]
It is furthermore clear that the ideal of $n \times n$ - minors of each of these two blocks is $\Jac_g$. But this in particular proves (\ref{jac4}), as required.
\end{proof}

\begin{cor} Notation as in the lemma. Suppose that $h$ furthermore log-resolves $\cI_{r,X}$ and $J_r$. Let $F$ be the divisor defined by
\[\cI_{r, \Delta} \cdot \cO_{\overline{\Delta}} = \cO_{\overline{\Delta}}(-F).\]
Then the following inequality is true.
\[K_{\overline{\Delta}/\Delta} - \frac{1}{r} F \leq (K_{A'/B} - n R_\Delta)_{|\overline{\Delta}}.\]
\end{cor}
\begin{proof} It follows from Lemma \ref{adjunction2} that
\[\Jac_h^r = (\cI_{r,X} \cdot \cO_{\overline{\Delta}}) \cdot \cO_{\overline{\Delta}}(-rK_{\overline{\Delta}/\Delta}),\]
and similarly for $\Jac_g^r$. Combining this with the inequality of ideals
\[(\Jac_{\pi} \cdot \cO_{\overline{\Delta}}) \cdot (\Jac_{g} \cdot \cO_{\overline{\Delta}})^2 \subseteq \Jac_h \cdot (\cO_A(-nR_\Delta)) \cdot \cO_{\overline{\Delta}}\]
we get the following inequality of divisors:
\[- (K_{A'/A})_{|\overline{\Delta}} - \frac{2}{r} F - 2 f^* K_{\Delta'/\Delta} \leq -\frac{1}{r} F - K_{\overline{\Delta}/\Delta} - (n R_\Delta)_{|\overline{\Delta}}.\]
Since the morphism $\rho : A \to B$ is the product $g \times g : X' \times X' \to X \times X$, we have that
\[2K_{\Delta'/\Delta} = (K_{A/B})_{|\Delta'}.\]
The inequality simplifies to
\[- (K_{A'/A})_{|\overline{\Delta}} - \frac{1}{r} F - f^* (K_{A/B})_{|\Delta'} \leq - K_{\overline{\Delta}/\Delta} - (n R_\Delta)_{|\overline{\Delta}}.\]
The corollary now follows easily.
\end{proof}

We can now finally prove our version of Takagi's subadditivity theorem.

\begin{thm}\label{subadditivity} Let $X$ be a $\Q$-Gorenstein variety and let $Z_1, Z_2$ be $\R_{>0}$-linear combinations of subschemes of $X$. Then
\[\overline{\Jac_X} \cdot \cJ(X, Z_1 + Z_2) \subseteq \cJ(X, Z_1) \cdot \cJ(X, Z_2).\]
\end{thm}
\begin{proof} It is enough to show that
\[\overline{\Jac_X} \cdot \cJ(X, Z_1 + Z_2) \subseteq \adj_{\Delta}(X \times X, p_1^*Z_1 + p_2^*Z_2) \cdot \cO_{\Delta}\]
where $p_1, p_2 : X \times X \to X$ are the two projections, since we have the easy inequality
\[\adj_{\Delta}(X \times X, p_1^*Z_1 + p_2^*Z_2) \subseteq \cJ(X \times X, p_1^*Z_1 + p_2^*Z_2)\]
and we will conclude by applying Lemma \ref{robslemma}. We let $\sigma : A' \to X \times X$ be the log-resolution as in Lemma \ref{prodjac} and we choose $g : X' \to X$ to also log-resolve $\Jac_X$, $\cI_{r,X}$, $J_r$ and $Z_1$ and $Z_2$, and $\sigma$ to be a log-resolution of all of these that is also a factorizing resolution for $\Delta$. Then, with the notation of the preceding corollary, we obtain that
\[K_{\overline{\Delta}/\Delta} - \frac{1}{r} F \leq (K_{A'/B} - n R_\Delta)_{|\overline{\Delta}}.\]
Let $G$ be the divisor defined by
\[\Jac_X \cdot \cO_{\overline{\Delta}} = \cO_{\overline{\Delta}}(-G).\]
Since $\cI_{r, \Delta}$ is an ideal that contains $\Jac_X^r$ (we do not need to perform this estimate, see the remark after the proof) we finally obtain the inequality
\[K_{\overline{\Delta}/\Delta} - G \leq (K_{A'/B} - n R_\Delta)_{|\overline{\Delta}}.\]
But by Lemma \ref{localvanishing} we have
\[\adj_\Delta(X \times X, p_1^*Z_1 + p_2^*Z_2) \cdot \cO_{\Delta} = h_*(\cO_{A'}(\lceil K_{A'/B} - \sigma^{-1} (p_1^*Z_1 + p_2^*Z_2) - n R_\Delta \rceil)_{|\overline{\Delta}}).\]
Putting this together with our inequality we are done.
\end{proof}

\begin{remark}\label{Strongersubadditivity} In fact, the proof also shows that
\[\langle \cI_{r,X} \rangle^{1/r} \cdot \cJ(X, Z_1 + Z_2) \subseteq \cJ(X, Z_1) \cdot \cJ(X, Z_2)\]
in the sense of Kawakita's $\Q$-ideals and his partial ordering on them, see \cite{KawakitaLogDisc}. Indeed, to obtain this we simply have to not perform the estimate in the proof that refers to this remark.
\end{remark}

\bibliography{paper}{}
\bibliographystyle{plain}

\end{document}